\documentclass[12pt]{amsart}
\usepackage{amscd,amssymb,subfigure}
\usepackage[all]{xy}
\usepackage{graphicx,verbatim}
\usepackage[colorlinks,plainpages,backref,urlcolor=blue]{hyperref}
\usepackage{pstricks}
\usepackage{mathtools}

\numberwithin{equation}{section}

\theoremstyle{plain}
\newtheorem{theorem}[subsection]{Theorem}

\newtheorem{corollary}[subsection]{Corollary}

\theoremstyle{definition}
\newtheorem{remark}[subsection]{Remark}

\newtheorem{example}[subsection]{Example}

\newtheorem*{ack}{Acknowledgment}

\renewcommand{\o}[1]{\overline{#1}}

\newcommand{\A}{\mathcal{A}}
\newcommand{\I}{\mathcal{I}}

\newcommand{\Ca}{\mathcal{C}}

\newcommand{\Z}{\mathbb{Z}}
\newcommand{\Q}{\mathbb{Q}}

\newcommand{\K}{\mathbb{K}}

\newcommand{\GG}{\Gamma}
\newcommand{\wX}{\widetilde{X}}
\newcommand{\wY}{\widetilde{Y}}
\newcommand{\wA}{\widehat{\A}}

\newcommand{\wR}{\widehat{\mathsf{R}}}

\newcommand{\OS}{\operatorname{A}^\bullet}
\newcommand{\DOS}{\operatorname{A}_{+}^{\bullet}}
\newcommand{\dOS}{\operatorname{A}_{+}}
\newcommand{\QOS}{\o{\operatorname{A}}^\bullet}
\newcommand{\Rr}{\mathsf{R}}

\DeclareMathOperator{\rank}{rank}
\DeclareMathOperator{\coker}{coker}
\DeclareMathOperator{\id}{id}

\DeclareMathOperator{\pr}{pr}
\DeclareMathOperator{\im}{im}
\DeclareMathOperator{\gr}{gr}

\begin{document}

\title[Nilpotent quotients of higher homotopy groups]%
{On the second nilpotent quotient of higher homotopy groups, for hypersolvable
arrangements}

\author[A.~D.~Macinic]{Daniela Anca~Macinic$^1$}
\address{Simion Stoilow Institute of Mathematics, 
P.O. Box 1-764,
RO-014700 Bucharest, Romania}
\email{Anca.Macinic@imar.ro}
\thanks{$^1$ Supported by a grant of the Romanian National Authority for Scientific
Research, CNCS  UEFISCDI, project number PN-II-RU-PD-2011-3-0149} 

\author[D.~Matei]{Daniel~Matei$^2$}
\address{Simion Stoilow Institute of Mathematics, 
P.O. Box 1-764,
RO-014700 Bucharest, Romania}
\email{Daniel.Matei@imar.ro}
\thanks{$^2$ Partially supported by
the Romanian Ministry of 
National Education, CNCS-UEFISCDI, grant PNII-ID-PCE-2012-4-0156}

\author[S.~Papadima]{Stefan~Papadima$^3$}
\address{Simion Stoilow Institute of Mathematics, 
P.O. Box 1-764,
RO-014700 Bucharest, Romania}
\email{Stefan.Papadima@imar.ro}
\thanks{$^3$ Partially supported by
the Romanian Ministry of 
National Education, CNCS-UEFISCDI, grant PNII-ID-PCE-2012-4-0156}



\date{October 17, 2013}

\begin{abstract}
We examine the first non-vanishing higher homotopy group, $\pi_p$, of the complement
of a hypersolvable, non--supersolvable, complex hyperplane arrangement,
as a module over the group ring of the fundamental group, $\Z\pi_1$.
We give a presentation for the $I$--adic completion of $\pi_p$.
We deduce that the second nilpotent $I$--adic quotient of $\pi_p$
is determined by the combinatorics of the arrangement, and we give a
combinatorial formula for the second associated graded piece, $\gr^1_I  \pi_p$.
We relate the torsion of this graded piece to the dimensions of the minimal generating
systems of the Orlik--Solomon ideal of the arrangement $\A$ in degree $p+2$, for various field coefficients.
When $\A$ is associated to a finite simple graph, we show that $\gr^1_I  \pi_p$
is torsion--free, with rank explicitly computable from the graph.
\end{abstract}

\maketitle

\tableofcontents

\section{Introduction}
\label{sec:intro}

\subsection{Overview}
\label{ssi1}

The {\em hypersolvable} class introduced in \cite{JP1}, \cite{JP2} is well adapted for 
homotopy computations with combinatorial flavour; see \cite{PS}, \cite{DP}.

Let $X$ be a path-connected space, with fundamental group $\pi_1 := \pi_1(X)$.
The higher homotopy groups of $X$ have a natural module structure over the group ring,
$\Rr := \Z \pi_1$. In general, their computation can be an extremely difficult problem.
When $X$ is not aspherical, homological methods may be used to tackle the first higher
non-trivial homotopy group, $\pi_p := \pi_p(X)$, by Hurewicz. This $\Rr$-module is still very
hard to describe,  when $\pi_1$ is non-trivial. Let $I \subseteq \Rr$ be the augmentation ideal. 
A reasonable idea is to approximate $\pi_p$ by its {\em nilpotent quotients},
$\pi_p/I^q \pi_p$ (for $q\ge 1$), or by the {\em associated graded} module
over $\gr_I^\bullet \Rr$, $\gr_I^\bullet \pi_p := \oplus_{q\ge 0} (I^q \pi_p/I^{q+1} \pi_p)$.

Now, let $\A$ be a central, hypersolvable, complex hyperplane arrangement, with affine complement denoted $X$.
For homotopy computations on $X$,
we may also assume $\A$ is essential. As shown in \cite{JP1}, $X$ is aspherical if and only if 
$\A$ is fiber-type (supersolvable). So, we also assume that $\A$ is not supersolvable.

Let $p(\A)$ be the order of $\pi_1$-connectivity 
of $X$, introduced in \cite{PS}, and let $r(\A)$ be the rank of the arrangement. 
We know that $2 \leq p \coloneqq p(\A) < r \coloneqq r(\A)$, and both $p$ and $r$ are combinatorial
(i.e., they depend only on the intersection lattice of $\A$). According to  \cite{PS},
$\pi_p=\pi_p(X)$ is the first higher non--trivial homotopy group of $X$.
It is also known that both 
$\gr_I^\bullet \Z \pi_1 $ and the first graded piece (nilpotent quotient) $\gr_I^0 \pi_p $ are 
combinatorial and torsion-free.

In \cite{DP}, the case when $p$ is maximal, i.e., $p=r-1$, was analyzed. 
It turned out that the $\gr_I^\bullet \Z \pi_1 $--module $\gr_I^\bullet \pi_p$  
is torsion-free, given by an explicit combinatorial  formula. 
Unfortunately, this formula does not hold, in general.

Here, we aim at removing the additional hypothesis on $p$, 
and see what can be said about $\pi_p$. 

\subsection{Results}
\label{ssi2}

Set $\Rr_q \coloneqq \Z\pi_1 /\ I^q$, for $1\le q < \infty$, and $\Rr_ \infty \coloneqq     \widehat{\Z \pi_1}$, 
where $\widehat{\Z \pi_1}$ is the $I$--adic completion of $\Z \pi_1$.
The first main results converge to a convenient $\Rr_q$--presentation of 
$\pi_p \otimes_{\Z\pi_1} \Rr _q$, for $q \leq \infty$. 
These are given in Theorem \ref{th:cpi} (for $q=\infty$) and Corollaries \ref{cor:Ipi}, \ref{cor:Iqpi} (for $q<\infty$). 
Note that $\pi_p \otimes_{\Z\pi_1} \Rr _q$ is the $q$-th nilpotent quotient, $\pi_p/I^q \pi_p$, for $q< \infty$.
When $q=2$, both  the second nilpotent quotient 
$\pi_p \otimes_{\Z\pi_1} \Rr _2$ and the second graded piece
$\gr ^1_I \pi_p$ have an explicit combinatorial formula, derived in Theorem \ref{th:cpi2}.

The second type of main results is related to torsion in $\gr_I^1 \pi_p$. 
It turns out that this problem leads to a basic question in combinatorial 
arrangement theory; compare  with \cite{OT}, \cite{Y}, \cite{SS}, \cite{CF}. 
Let $\Lambda^\bullet := \Lambda^\bullet (\A)$ be the exterior algebra over $\Z$
generated by the set of hyperplanes of an arbitrary arrangement $\A$.
Let $\I^\bullet := \I^\bullet (\A) \subseteq \Lambda^\bullet $ be the Orlik-Solomon 
ideal of $\A$, and denote by $\OS  (\A)=\Lambda /\I$ the 
Orlik-Solomon algebra over $\Z$, known to be torsion-free. 
By a celebrated result of Orlik and Solomon, the $\K$--specialization $\OS  (\A)_{\K}$
is isomorphic to the $\K$--cohomology ring of the affine complement of $\A$, for every
commutative ring $\K$.

Let $\Lambda^{+} \I \subseteq \I$ be the decomposable Orlik-Solomon ideal. We introduce 
$\DOS(\A) :=\Lambda /\Lambda^{+} \I$, the {\em decomposable Orlik-Solomon algebra}. 
Is $\DOS(\A)$ also torsion-free? At the time of writing, we have no example 
where torsion appears. When $\A$ is hypersolvable and not supersolvable, we show in Theorem 
\ref{th:tors} that $\gr_I^1 \pi_p$ is torsion-free precisely when $\dOS^{p+2} (\A)$  has no torsion.

Consider the quadratic Orlik-Solomon algebra, $\QOS (\A) \coloneqq \Lambda /\I_2$, 
where $\I_2$ is the ideal generated by $\I^2$, the degree $2$ component of $\I$. 
When $\A$ is supersolvable, it is known that $\OS(\A) = \QOS (\A)$, 
see \cite{SY}; hence, in this case, $\DOS (\A)$ has no torsion.

When $\A$ is hypersolvable, not supersolvable, and {\em graphic}
(i.e., a subarrangement of a braid arrangement, associated to a finite simple graph), 
we prove in Corollary \ref{gr1_graphic} that $\gr_I^1 \pi_p$ has a simple description in terms of the graph, 
in particular it has no torsion. The graph from Example \ref{extension_result} shows that this combinatorial description may
be done outside the maximal range $p=r-1$ from \cite{DP}. 

\subsection{Questions}
\label{ssi3}

We are left with  some open questions concerning hypersolvable arrangements. 
Is $\DOS (\A)$  torsion-free, at least in degree $\bullet=p+2$? 
What if we restrict the question to {\it 2--generic} arrangements
(i.e., arrangements with no collinearity relations, known to be hypersolvable)?
See also Remark \ref{q:arbtors} on arbitrary arrangements.

\section{A preliminary module presentation}
\label{sec:prel}

We shall work in the context of \cite[Sections 5 and 6]{DP}. Let $\A$ be a hypersolvable 
complex hyperplane arrangement which is not supersolvable, and $X=M'(\A)$ 
its complement in affine space.

We know that $\A$ is a $p$-generic section of its supersolvable deformation, $\wA$.
Set $Y=M'(\wA)$, and let $j:X\hookrightarrow Y$ denote the inclusion. Denote by $\pi_1$ 
the fundamental groups identified through the induced map $j_{\sharp}:\pi_1(X)\to\pi_1(Y)$.
Let $\tilde{j}:\wX\to\wY$ be the $\pi_1$-equivariant map induced on universal covers. 
Denote by $\tilde{j}_\bullet: C_\bullet(\wX)\to C_\bullet(\wY)$ the $\Z\pi_1$-linear 
chain map between the $\pi_1$-equivariant cellular chains on the universal covers, 
and by $j_\bullet:H_\bullet(X)\to H_\bullet(Y)$ the induced map in integral homology.

We have split exact sequences of finitely generated free abelian groups,
\begin{equation}\label{eq:hp}
0\to H_\bullet(X)\xrightarrow{j_\bullet}H_\bullet(Y)\xrightarrow{\Pi_\bullet}
H_\bullet(Y,X)\to 0,
\end{equation}
whose duals,
\begin{equation}\label{eq:chp}
0\to H^\bullet(Y,X)\xrightarrow{\Pi^\bullet}H^\bullet(Y)\xrightarrow{j^\bullet}
H^\bullet(X)\to 0,
\end{equation}
may be described in purely combinatorial terms: $j^\bullet$ may be identified with the 
canonical surjection,
\begin{equation}\label{eq:osm}
j^\bullet:\QOS(\A)\twoheadrightarrow\OS(\A),
\end{equation}
between Orlik-Solomon algebras.

For simplicity, in the sequel we set $\Rr \coloneqq \Z\pi_1$. Note that 
$C_\bullet(\wX)=H_\bullet(X)\otimes \Rr$ and $C_\bullet(\wY)=H_\bullet(Y)\otimes \Rr$, 
as $\Rr$-modules, by the minimality property for arrangement complements \cite[Corollary 6]{DP}.

Denoting by $\tilde\partial_{\bullet}:C_\bullet(\wY)\to C_{\bullet-1}(\wY)$ 
the differential on the equivariant chain complex of $\wY$, we have the following.

\begin{theorem}\emph{(\cite{DP})}\label{th:pip}
The $\Rr$-module $\pi_p$ is isomorphic to the cokernel of the $\Rr$-linear map
\[
\tilde\partial_{p+2} + \tilde j_{p+1}:(H_{p+2}Y\oplus H_{p+1}X)\otimes \Rr\to H_{p+1}Y\otimes \Rr.
\]
\end{theorem}

Due to $\Rr$-linearity, $\tilde j_{\bullet}$ respects the $I$-adic filtrations, 
i.e., it sends $H_\bullet X\otimes I^q$ into $H_\bullet Y\otimes I^q$, for all $q$. 
The associated graded $\gr^{\bullet}_I \Rr$--linear map,
\[
\gr^{\bullet}\tilde j_{\bullet}:H_\bullet X\otimes \gr^{\bullet}_I \Rr\to 
H_\bullet Y\otimes \gr^{\bullet}_I \Rr \, ,
\] 
is equal to $j_{\bullet}\otimes\id$, by minimality. 
 
Similar considerations are valid for $\tilde\partial_\bullet$: by minimality again, 
it sends $H_\bullet Y\otimes I^q$ into $H_{\bullet-1} Y\otimes I^{q+1}$, for all $q$. 
The associated graded $\gr^{\bullet}_I \Rr$--linear map is denoted 
\[
E_1^\bullet\tilde\partial_\bullet: H_\bullet Y\otimes \gr^{\bullet}_I \Rr\to
H_{\bullet-1} Y\otimes \gr^{\bullet+1}_I \Rr.
\]

To describe the action of $E_1^\bullet\tilde\partial_\bullet$ on the free 
$\gr^{\bullet}_I \Rr$-generators, $H_\bullet Y\otimes 1$, 
we recall that $\gr^0_I \Rr=\Z\cdot 1$, and $\gr^1_I \Rr$ is naturally identified 
with $(\pi_1)_{ab}=H_1(Y)$. We denote by $H_1$ both $H_1(X)$ and $H_1(Y)$,
identified via $j_1$. 

Now it follows from \cite[Section 6]{DP} that the restriction of 
$E_1^\bullet\tilde\partial_\bullet$ to 
$H_\bullet Y\equiv H_\bullet Y\otimes 1\subseteq H_\bullet Y\otimes\gr^0_I \Rr$, 
denoted 
\begin{equation}\label{eq:del}
\partial_\bullet: H_\bullet Y\to H_{\bullet-1} Y\otimes H_1,
\end{equation}
has dual, up to sign,
\begin{equation}\label{eq:delst}
\partial_\bullet^*: \o{\operatorname{A}}^{\bullet-1}(\A)\otimes \o{\operatorname{A}}^1 (\A) \to\QOS(\A),
\end{equation}
given by the multiplication of the quadratic OS-algebra.

The description \eqref{eq:del} of $E_1^\bullet\tilde\partial_\bullet$  is related to 
the spectral sequence associated to the equivariant chain complex of a $CW$-complex, 
analyzed in full generality in \cite{PS2}.

\section{Completion of the presentation}
\label{sec:compl}

In this section we pursue the following idea: Use completion constructions 
to simplify the presentation in Theorem \ref{th:pip}, more exactly, to 
replace $\tilde j_{p+1}$ by $j_{p+1}\otimes\id$, without altering 
$E_1^\bullet\tilde\partial_{p+2}$. We refer the reader to \cite[Chapitre III.2]{B} for
standard completion techniques.

We explain now how these work concretely. The ring $\wR$ is endowed with the
canonical, decreasing, complete, separated, and multiplicative filtration
$\{F^q\}_{q\ge 0}$, as $\wR=\varprojlim \Rr/I^q$. In addition, $\wR/F^q=\Rr/I^q$
and $\gr_F^q\wR=\gr_I^q \Rr$, for all $q$. Every right $\wR$-module $M$ 
has the canonical filtration $\{M\cdot F^q\}_{q\ge 0}$, and $\wR$-linear maps
preserve canonical filtrations. Furthermore, we have the following convenient
test, for an $\wR$-linear map $f$ between complete and separated modules: $f$
is an isomorphism if and only if $\gr_F^{\bullet}(f)$ is an isomorphism. These
facts will lead to the first property of the aforementioned replacement.

For the second property, let us notice that, given an arbitrary map in $\wR$-$Mod$,
$f:M\to N$, we have that $f(M\cdot F^q)\subseteq N\cdot F^{q+1}$ for all $q$
if and only if $\gr_F^{\bullet}(f)=0$. If this happens, $f$ induces a 
$\gr_F^{\bullet}\wR$-linear map, 
$E_1^{\bullet}f:\gr_F^{\bullet}M\to\gr_F^{\bullet+1}N$. 

Finally, there is the completion functor, 
$\widehat{(\cdot)}:\Rr$-$Mod\to\wR$-$Mod$, given by
$M\mapsto\widehat{M}=\varprojlim (M/M\cdot I^q)$. 
On free finitely generated $\Rr$-modules, $\widehat{(\cdot)}$ is naturally equivalent
with $(\cdot)\otimes_\Rr\wR$. More precisely, if $M=H\otimes \Rr$, where $H$ is a
finitely generated free abelian group, then $M\otimes_\Rr\wR=H\otimes \wR$, with
canonical (complete and separated) filtration $\{H\otimes F^q\}_{q\ge 0}$.
Clearly, $\gr_F^{\bullet}(H\otimes \wR)=\gr_I^{\bullet}(H\otimes \Rr)= H\otimes \gr_I^{\bullet} \Rr$.

The (decreasing, multiplicative) $I$-adic filtration $\{I^q\}_{q\ge 0}$ of $\Rr$
leads to similar constructions, 
$\gr^{\bullet}(\varphi):\gr_I^{\bullet}M\to\gr_I^{\bullet}N$
(for $\varphi:M\to N$ $\Rr$-linear), respectively 
$E_1^{\bullet}(\varphi) :\gr_I^{\bullet}M\to\gr_I^{\bullet+1}N$, when 
$\gr^{\bullet}(\varphi)=0$. When both $M$ and $N$ are finitely generated free
$\Rr$-modules, $\gr_F^{\bullet}(\varphi\otimes_\Rr\wR)=\gr^{\bullet}(\varphi)$.
If in addition $\gr^{\bullet}(\varphi)=0$, then 
$E_1^{\bullet}(\varphi\otimes_\Rr\wR)=E_1^{\bullet}(\varphi)$.

\begin{theorem}\label{th:cpi}
Let $\A$ be a hypersolvable and not supersolvable arrangement. 
Then the $\wR$-module $\pi_p\otimes_\Rr\wR$ is isomorphic to the cokernel of an $\wR$-linear map
\[
D_{p+2}:H_{p+2}Y\otimes\wR\to H_{p+1}(Y,X)\otimes\wR,
\]
with the property that $\gr_F^{\bullet}(D_{p+2})=0$ and 
$E_1^{\bullet}(D_{p+2}):H_{p+2}Y\otimes\gr_I^{\bullet}\Rr \to H_{p+1}(Y,X)\otimes\gr_I^{\bullet+1}\Rr$
acts on the free $\gr_I^{\bullet}\Rr$-generators by
\[
H_{p+2}Y\xrightarrow{\partial_{p+2}}H_{p+1}Y\otimes H_1
\xrightarrow{\Pi_{p+1}\otimes\id}H_{p+1}(Y,X)\otimes H_1,
\]
where $\partial_{p+2}$ is described in \eqref{eq:del}-\eqref{eq:delst},
and $\Pi_{p+1}$ is defined in \eqref{eq:hp} and \eqref{eq:chp}.
\end{theorem}

\begin{proof}
Choose a splitting in \eqref{eq:hp}, 
$\sigma_{\bullet}:H_{\bullet}(Y,X)\hookrightarrow H_{\bullet}Y$.
The $\Rr$-presentation from Theorem~\ref{th:pip} gives a presentation
for $\pi_p\otimes_\Rr\wR$ as the cokernel of the $\wR$-linear map
\begin{equation}\label{eq:cpres}
\tilde\partial_{p+2}\otimes_\Rr\wR + \tilde j_{p+1}\otimes_\Rr\wR 
:(H_{p+2}Y\oplus H_{p+1}X)\otimes\wR\to (H_{p+1}X\oplus H_{p+1}(Y,X))\otimes\wR.
\end{equation}
Consider the $\wR$-linear map
\begin{equation}\label{eq:csig}
\tilde j_{p+1}\otimes_{\Rr} \wR + \sigma_{p+1}\otimes\id_{\wR}: 
(H_{p+1}X\oplus H_{p+1}(Y,X))\otimes\wR\to(H_{p+1}X\oplus H_{p+1}(Y,X))\otimes\wR.
\end{equation}
Since $\gr^{\bullet}\tilde j_{p+1}=j_{p+1}\otimes\id$ and 
$\gr_F^{\bullet}(\sigma_{p+1}\otimes \id_{\wR})=\sigma_{p+1}\otimes\id$,
we infer that \eqref{eq:csig} is an isomorphism, 
by $\wR$-completeness and separation. 

Hence, 
$H_{p+1}Y\otimes\wR\cong \im(\tilde j_{p+1}\otimes_{\Rr} \wR)\oplus\im(\sigma_{p+1}\otimes\id_{\wR})$,
and $H_{p+1}Y\otimes\wR/\im(\tilde j_{p+1}\otimes_{\Rr} \wR)\cong H_{p+1}(Y,X)\otimes\wR$.
Moreover, $\gr_F^{\bullet}(H_{p+1}Y\otimes\wR\xrightarrow{\pr_{p+1}}
H_{p+1}Y\otimes\wR/\im(\tilde j_{p+1}\otimes_{\Rr} \wR))$ is identified with
$H_{p+1}Y\otimes\gr_I^{\bullet}\Rr\xrightarrow{\Pi_{p+1}\otimes\id}
H_{p+1}(Y,X)\otimes\gr_I^{\bullet}\Rr.$

Set $D_{p+2}=\pr_{p+1}\circ(\tilde\partial_{p+2}\otimes_\Rr\wR)$.
Combining \eqref{eq:cpres} and \eqref{eq:csig} we obtain that 
$\pi_p\otimes_\Rr\wR\cong\coker(D_{p+2})$, and 
$\gr_F^{\bullet}(D_{p+2})=(\Pi_{p+1}\otimes\id)\circ\gr^{\bullet}\tilde\partial_{p+2}=0$.
The assertion on $E_1^{\bullet}(D_{p+2})$ follows from \eqref{eq:del}.
\end{proof}

It is now an easy matter to derive $\Rr_q$-presentations for 
$\pi_p\otimes_\Rr \Rr_q=\pi_p/I^q\cdot\pi_p$, for all $1\le q<\infty$.
Note that $\gr_I^s \Rr_q=\gr_I^s \Rr$ for $s<q$, and $\gr_I^s \Rr_q=0$ for $s\ge q$.
Note also that $H_{p+1}(Y,X)\ne 0$, by the definition of $p(\A)$.

\begin{corollary}\emph{(\cite{PS})}\label{cor:Ipi}
If $\A$ is a hypersolvable and not supersolvable arrangement, then
$\gr_I^0 \pi_p= \pi_p/I\cdot\pi_p=H_{p+1}(Y,X)$ does not vanish.
\end{corollary}

\begin{example}
Note that the hypersolvability hypothesis on $\A$ is crucial. 
Indeed, recall from \cite{PS} that by definition $p=p(M'(\A))$ is equal to 
$\sup\{s\mid\dim_{\Q}H_t(M'(\A),\Q)=\dim_{\Q}H_t(\pi_{1}M'(\A),\Q),\forall t\le s\}$.
When $\A$ is hypersolvable, this is equal to 
$p(\A) \coloneqq \sup\{s\mid \rank \operatorname{A}^t(\A)=
\rank \o{\operatorname{A}}^t(\A),\forall t\le s\}$.

Now, let $\A$ be the aspherical Coxeter arrangement of type $D_n, n\ge 4$.
Since the Orlik-Solomon algebra $\OS(\A)$ is not quadratic \cite{F},
$\A$ is not supersolvable \cite{SY} and $2\le p(\A)<\infty$. Clearly $\A$
cannot be hypersolvable, since $\pi_p(M'(\A))=0$.
\end{example}

\begin{corollary}\label{cor:Iqpi}
Let $\A$ be a hypersolvable and not supersolvable arrangement and $q\ge 2$. 
Then $\pi_p/I^q\cdot\pi_p$ is isomorphic over $\Rr_q$ with 
$\coker(D_{p+2}\otimes_{\wR}\Rr_q:H_{p+2}Y\otimes \Rr_q\to H_{p+1}(Y,X)\otimes \Rr_q)$.
Furthermore, $\gr_I^{\bullet}(D_{p+2}\otimes_{\wR}\Rr_q)=0$, and 
$E_1^s(D_{p+2}\otimes_{\wR}\Rr_q):H_{p+2}Y\otimes\gr_I^s \Rr \to H_{p+1}(Y,X)\otimes\gr_I^{s+1}\Rr$
is equal to $E_1^s(D_{p+2})$, for $s<q-1$, and it is $0$, for $s=q-1$.
\end{corollary}

\begin{proof}
Tensor the $\wR$-presentation from Theorem \ref{th:cpi}, over $\wR$, with $\wR/F^q=\Rr_q$.
The claims on $\gr^{\bullet}$ and $E_1^{\bullet}$ follow from the fact that 
$\gr^{\bullet}\Rr_q=\gr^{\bullet}\Rr/\gr^{\ge q}\Rr$.
\end{proof}

When $q=2$, everything becomes explicit. The exact sequence
\begin{equation}\label{eq:qtwo}
0\to I/I^2\to \Rr/I^2\to \Rr/I \to 0
\end{equation}
has a canonical splitting. Hence, $\Rr_2=\Z\cdot 1\oplus H_1$, where $H_1$
is free abelian, of rank $|\A|$. The $I$-adic filtration is given by
$I^0\cdot \Rr_2=\Rr_2$, $I\cdot \Rr_2=H_1$ and $I^2\cdot \Rr_2=0$. Hence, the filtered
ring $\Rr_2$ is combinatorially determined.

The map $D_{p+2}\otimes_{\wR} \Rr_2: H_{p+2}Y\otimes(\Z\cdot 1\oplus H_1)\to 
H_{p+1}(Y,X)\otimes(\Z\cdot 1\oplus H_1)$ is zero on $H_{p+2}Y\otimes H_1$;
on $H_{p+2}Y\otimes 1\equiv H_{p+2}Y$,
it is equal to $(\Pi_{p+1}\otimes\id)\circ\partial_{p+2}:H_{p+2}Y\to H_{p+1}(Y,X)\otimes H_1$.
Hence, the filtered $\Rr_2$-module 
$\pi_p/I^2\cdot\pi_p$ is combinatorially determined, see \eqref{eq:osm} and \eqref{eq:delst}. In particular,
$\gr_I^1\pi_p$ is combinatorially determined. We will need an explicit
combinatorial description of the second graded piece, $\gr_I^1\pi_p$.
By \eqref{eq:chp} and \eqref{eq:osm}, $\Pi_{p+1}^*:H^{p+1}(Y,X)\to H^{p+1}Y$
is the inclusion,
\begin{equation}\label{eq:Piq2}
\Pi_{p+1}^*:(\I/\I_2)^{p+1}\hookrightarrow(\Lambda/\I_2)^{p+1}.
\end{equation}

We infer from \eqref{eq:delst} that (up to sign) 
\begin{equation}\label{eq:delq2}
\partial_{p+2}^*:(\Lambda/\I_2)^{p+1}\otimes\Lambda^1\to(\Lambda/\I_2)^{p+2}
\end{equation}
is induced by the multiplication map $\mu$ of $\Lambda^{\bullet}$. We thus obtain the following
explicit combinatorial description:
\begin{equation}\label{eq:delPiq2}
\partial_{p+2}^*\circ(\Pi_{p+1}\otimes\id)^*
:(\I/\I_2)^{p+1}\otimes\Lambda^1\xrightarrow{\pm\mu}(\Lambda/\I_2)^{p+2}.
\end{equation}

Using the $\Rr_2$-presentation from Corollary \ref{cor:Iqpi}, we deduce that 
$\gr_I^1\pi_p$ is given by
\begin{align*}
\gr_I^1\pi_p 
&=I\cdot(\pi_p/I^2\cdot\pi_p)=H_{p+1}(Y,X)\otimes H_1/\im(D_{p+2}\otimes_{\wR} \Rr_2)\cap 
(H_{p+1}(Y,X)\otimes H_1) \\
&=\coker((\Pi_{p+1}\otimes\id)\circ\partial_{p+2}:H_{p+2}Y\to H_{p+1}(Y,X)\otimes H_1).
\end{align*}

We summarize our results for $q=2$ as follows.

\begin{theorem}\label{th:cpi2}
Let $\A$ be a hypersolvable and not supersolvable arrangement, and $p=p(\A)$. 
Then the second nilpotent quotient, $\pi_p M'(\A)/I^2\cdot\pi_p M'(\A)$ is 
combinatorially determined as a filtered $\Z\pi_1 M'(\A)/I^2$-module.
The finitely generated abelian group $\gr_I^1\pi_p M'(\A)$ is also 
combinatorially determined, with $\Z$-presentation
\[
\gr_I^1\pi_p M'(\A)=
\coker(H_{p+2}Y\xrightarrow{(\Pi_{p+1}\otimes\id)\circ\partial_{p+2}}H_{p+1}(Y,X)\otimes H_1). 
\]
\end{theorem}

\section{Torsion issues}
\label{sec:tors}

In this section, we analyze the torsion of the second graded piece of $\pi_p$.

\begin{theorem}\label{th:tors}
Let $\A$ be a hypersolvable and not supersolvable arrangement, and $p=p(\A)$. Then the following are equivalent:
\begin{enumerate}
\item The second graded piece, $\gr^1_I \pi_p(M'(\A))$, has no torsion.
\item The decomposable Orlik-Solomon algebra, $\DOS(\A)$, is free in degree $\bullet= p+2$.
\item The graded abelian group of indecomposable OS--relations, $(\I /\Lambda^+ \I)^\bullet$ is free in degree $\bullet= p+2$.
\end{enumerate}
\end{theorem}

\begin{proof}
Let $\K$ be a field. We infer from Theorem \ref{th:cpi2} and \eqref{eq:delPiq2} that the $\K$-dual $(\gr^1_I \pi_p) \otimes \K^*$ is isomorphic to 
$\ker (\mu: (\I/\I_2)^{p+1} \otimes \Lambda^1 \rightarrow (\Lambda/\I_2)^{p+2})_{\K}$ over $\K$, where
the subscript $\K$ denotes specialization to $\K$--coefficients. 
Since $\I_2(\A)_{\K}=\I(\hat{\A})_{\K}$ (\cite{SY, JP1}),
both Hilbert series, $( \I /\I_2)^\bullet \otimes \Lambda^1 (t)$ and $(\Lambda /\I_2)^{\bullet}(t)$, are independent of $\K$,
taking into account that Orlik-Solomon algebras are torsion-free \cite{OT}.

Hence, $\gr^1_I \pi_p$ is free if and only if $\dim_{\K}\coker(\mu)_{\K}$ is independent of $\K$, in degree $p+2$. Plainly, 
$\coker(\mu)_{\K}^{p+2}= \dOS^{p+2} (\A) \otimes \K$. Therefore, $(1) \Leftrightarrow (2)$. The split exact sequence
\[
0 \rightarrow (\I /\Lambda^+ \I)^\bullet \rightarrow \DOS (\A) \rightarrow \OS (\A) \rightarrow 0
\]
gives the equivalence $(2) \Leftrightarrow (3)$. 
\end{proof}

In what follows, the subscript $\K$ denotes OS--type objects with coefficients in $\K$.
For an arbitrary arrangement $\A$, set 
$\operatorname{A}_{\K}^{\bullet}(\A)(t) \coloneqq \sum_{m \geq 0}b_m(\A)t^m$; 
this Hilbert series is independent of the field $\K$. Define 
\[
(\I/\Lambda^+\I)^\bullet_{\K}(t) \coloneqq \sum_{m \geq 2}
r_m (\A)_{\K} t^m=(\DOS)_{\K}(\A)(t)-\OS_{\K}(\A)(t).
\]
When we write $r_m(\A)$, we mean that $r_m (\A)_{\K}$ is independent of $\K$.
With this notation, we extract from the proof of Theorem \ref{th:tors} the following.

\begin{corollary}
\label{rank}
Assume that $\A$ satisfies the equivalent properties from Theorem \ref{th:tors}. 
Then $\gr_I^1 \pi_p M'(\A)$ is free, of rank 
$$ |\A|(b_{p+1}\widehat{\A}-b_{p+1} \A)- (b_{p+2}\widehat{\A}-b_{p+2} \A)+ r_{p+2} \A \, ,$$
where $\widehat{\A}$ is the supersolvable deformation of $\A$, constructed in \cite{JP1, JP2}.
\end{corollary}

\begin{example}
\label{supersolv}
If $\A$ is supersolvable, then $\DOS (\A)$ has no torsion. Indeed, in this case 
$\OS _{\K}(\A)=\QOS _{\K}(\A)$, according to \cite[Lemma 4.3]{SY}. We deduce that the Hilbert series 
$(\I /\Lambda^+ \I)^\bullet _{\K}(t)=(\I_2 /\Lambda^+ \I_2)^\bullet _{\K}(t)=(\dim_{\K}\I^2_{\K}) \cdot t^2$ is independent of $\K$.

When $\A$ is hypersolvable and $p(\A)=r(\A)-1$, then $\A$ is not supersolvable and $\dOS^{p+2}(\A)$ has no torsion; 
see \cite[Theorem 23]{DP} and Theorem \ref{th:tors}. This happens for instance when $\A$ is hypersolvable 
and $r(\A)=3$.
\end{example}

For the next examples, we need to review some standard constructions and terminology in arrangement theory.  
A subset $C \subseteq \A$ belongs to $\Ca_q(\A)$ (the set of $q$--circuits of $\A$) if and only if $|C|=q$ and the hyperplanes in $C$ 
form a minimally dependent set. We say that $C \subseteq \A$ has a chord, $c \in \A \setminus C$, if there is a partition, 
$C= C' \sqcup C''$, such that both $C' \cup \{c\}$ and $C'' \cup \{c\}$ are dependent subsets.  
Let $\Ca^{NC}_q(\A) \subseteq \Ca_q(\A)$ be the subset of chordless $q$--circuits.

Recall that $\Lambda_{\K}^\bullet$ is the exterior $\K$--algebra on $\A$, as usual. Denote by 
$\delta: \Lambda_{\K}^\bullet \rightarrow \Lambda_{\K}^{\bullet -1}$ the unique degree $-1$ graded algebra derivation taking 
the values $\delta(i)=1$, for all free algebra generators $i \in \A$. Note that $\delta^2=0$. 
For $C=\{i_1, \dots, i_q\} \subseteq \A, \; |C|=q$, denote by $e_C \in \Lambda^q$ the exterior monomial 
$i_1 \cdots i_q$ (which is well defined, up to a sign).

We recall that the Orlik-Solomon ideal $\I$ is generated by $\delta(e_C), \;C \in \Ca_{\bullet}(\A)$. It follows that
$\delta_q : \K-{\rm span} \langle e_C| \; C \in \Ca_{q+1}(\A)\rangle \rightarrow \I_{\K}^q$ induces a surjection, 
$\overline{\delta}_q : \K-{\rm span} \langle e_C| \; C \in \Ca_{q+1}(\A)\rangle \twoheadrightarrow (\I /\Lambda^+ \I)^q_{\K}$, 
for all $q$. The proof of Lemma 2.1 from \cite{CF} shows that the restriction
\begin{equation}
\label{delta_bar}
\overline{\delta}_q : \K-{\rm span}
\langle e_C| \; C \in \Ca_{q+1}^{NC}(\A)\rangle \twoheadrightarrow (\I /\ \Lambda^+ \I)^q_{\K}
\end{equation}
is still onto, for all $q$.

Now, let $\A_{\GG}$ be the {\em graphic arrangement} (see \cite{OT}) associated to the finite simple graph $\GG$, 
with hyperplanes in one to one correspondence with the edges of $\GG$. In this case, the $q$-circuits of $\A_{\GG}$ 
correspond to the $q$-cycles of $\GG$; furthermore, a circuit has a chord if and only if the corresponding cycle has a chord, 
in the sense of graph theory. By \cite[Lemma 6.2]{SS}, the map \eqref{delta_bar} is an isomorphism, for all $q$, when $\A=\A_{\GG}$.

\begin{corollary}
\label{gr1_graphic}
Let $\A=\A_{\GG}$ be hypersolvable and not supersolvable. Then the second graded piece, 
$\gr_I^1 \pi_p M'(\A_{\GG})$, is free, with rank given in 
Corollary \ref{rank}, where $r_{p+2} \A_{\GG}=|\Ca_{p+3}^{NC}(\A_{\GG})|$. Moreover, this rank may be explicitly 
computed from a hypersolvable composition series of the graph $\GG$.
\end{corollary}

\begin{proof}
The first assertion follows at once from the isomorphism \eqref{delta_bar}.
As for the second claim, let us examine the Betti numbers, $b_{\bullet}(\A_{\GG})$ and $b_{\bullet}(\widehat{\A}_{\GG})$,
appearing in Corollary \ref{rank}. We know that the Hilbert series of $\OS (\A_{\GG})$ can be computed from the 
chromatic polynomial of $\GG$ \cite{OT}. Finally, the Hilbert series of $\OS (\widehat{\A_{\GG}})$ is determined
by the exponents of a hypersolvable composition series of  $\GG$ \cite{PS}.
\end{proof}

 \vskip .1in
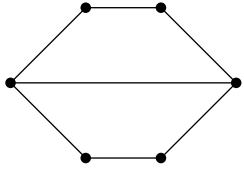
 \begin{pspicture}(10,4)

 \pscircle*(7,1){.07}
 \pscircle*(8,1){.07}
 \pscircle*(6,2){.07}
 \pscircle*(9,2){.07}
 \pscircle*(7,3){.07}
 \pscircle*(8,3){.07}
   
 \psline[linewidth=.5pt](7,1)(6,2)(7,3)(8,3)(9,2)(8,1)(7,1)
 \psline[linewidth=.5pt](6,2)(9,2)
 
 \end{pspicture}

\begin{example}
\label{extension_result}
The graphic arrangement $\A_{\GG}$ associated to the above graph $\GG$ (without triangles) is hypersolvable and not supersolvable, 
with $p(\A_{\GG})=2$ and $\rank(\A_{\GG})=5 > p+1$. Theorem 23 from \cite{DP} cannot be used, but $\gr_I^1 \pi_p M'(\A_{\GG})$ 
can be computed with the aid of Corollary \ref{gr1_graphic}.
\end{example}

\begin{remark}
\label{2-gen}
For a dependent arrangement (i.e., not boolean) define $c(\A)$ to be the smallest integer $q$ for which there is $C \subseteq \A$ 
dependent with $|C|=q$. Equivalently, $\Ca_{c(\A)}(\A) \neq \emptyset$, but $\Ca_{<c(\A)}(\A) = \emptyset$. Note that $c(\A) \geq 3$.  
We recall that an arrangement $\A$ is called $2$--generic when $c \coloneqq c(\A) > 3$. This implies that $\A$ is hypersolvable 
and not supersolvable, of a particular kind: $\pi_1 M'(\A)=\Z^{\A}, \; \QOS(\A)=\Lambda^\bullet, \; p=c-2$. 
Question: is $r_c (\A)_{\K}$ independent  of $\K$?
\end{remark}

\begin{example}
\label{free_gr1}
For an arbitrary arrangement $\A, r_{p+2}(\A)=0$ if $\Ca^{NC}_{p+3}(\A)= \emptyset$ (see \eqref{delta_bar}). When 
$\A$ is hypersolvable and not supersolvable (as in the $2$--generic example \ref{extension_result}), $\gr_I^1 \pi_p$ is free, 
with rank computed in Corollary \ref{rank}.
\end{example}

\begin{remark}
\label{q:arbtors}
Let $\A$ be an arbitrary arrangement. Note that $r_2(\A)_{\K}$ is independent of $\K$, and
$r_m(\A)=0$ for $m>r(\A)$. The first claim is immediate, since $(\I /\ \Lambda^+ \I)^2= \I^2$.
The second assertion follows from \eqref{delta_bar}, since plainly $\Ca_{m+1}(\A)=\emptyset$
for $m>r(\A)$. Question: is $r_m (\A)_{\K}$ independent  of $\K$, for $2<m\le r(\A)$?
\end{remark}

\begin{example}
\label{delta_not_inj}
In the graphic case, the computation of $r_{p+2}(\A)$ from Corollary \ref{gr1_graphic} is a consequence of the 
following two facts that hold on $\K$-span$\langle e_C|\; C \in \Ca^{NC}_{q+1}(\A) \rangle$:
\begin{enumerate}
\item $\ker (\o \delta_q)=\ker (\delta_q)$
\item $\delta_q$ is injective
\end{enumerate}
Assume $c \coloneqq c(\A)>3$. Then $\Ca_{c+1}(\A)=\Ca^{NC}_{c+1}(\A)$. Suppose moreover that $c=r(\A)$, 
and there is $C' \subseteq \A, \; |C'|=c+2$, such that   all $c$--element subsets of $C'$ are independent. 
(Clearly, this cannot happen when $\A$ is a graphic arrangement.) Then the above condition $(2)$ fails in degree $q=c$. 
Indeed, $\delta(e_{C'}) \in \K$-span$\langle e_C|\; C \in \Ca^{NC}_{c+1}(\A) \rangle$ is non--zero, and $\delta^2(e_{C'})=0$. 

We give a simple rank $4$ example illustrating the previously described setting.

Consider the $2$--generic arrangement $\A$ in $\mathbb C^4$ of equation
\[
xyzt(x+y+2z)(x+y+z+t)(x+2y-z+4t)=0,
\]
with $c=4$. Denote by $H$ the hyperplane of equation $x+y+z+t=0$ and by $P$ the hyperplane of equation $x+2y-z+4t=0$. 
Then the  subset of hyperplanes $C'= \{x,y,z,t,H,P\}$ has the property that all its $4$--element subsets are
independent, as needed.
\end{example}

\begin{example}
\label{non-Fano}
If $r_q (\A)_{\K}\le 1$ and $|\Ca^{NC}_{q+1}(\A)|>1$, property $(1)$ from Example \ref{delta_not_inj}
is also violated. Indeed, it would imply that $\im (\delta_q)$ is at most one-dimensional, which clearly forces
$|\Ca^{NC}_{q+1}(\A)| \le1$. 

Here is a simple example where this happens. Let $\A$ be the $2$--generic arrangement in $\mathbb C^4$ of equation
\[
xyzt(x+y+z+t)(x-y-z+t)=0,
\]
with $c=4$. Denote by $H$ and $P$ the last two hyperplanes. It is easy to check that $\Ca_5^{NC} (\A)$ has
two elements, namely $C_5= \{x,y,z,t,H\}$ and $C'_5= \{x,y,z,t,P\}$. Since the subsets $C_4= \{y,z,H,P\}$ and
$C'_4= \{x,t,H,P\}$ are $4$--circuits and 
$\delta(e_{C_5})-\delta(e_{C'_5})= (e_x-e_t)\cdot \delta(e_{C_4})+ (e_y-e_z)\cdot \delta(e_{C'_4})$,
we infer from \eqref{delta_bar} that $r_4 (\A)_{\K}\le 1$, as needed.
\end{example}

\begin{ack}
We are grateful to Alex Suciu, for illuminating conversations about
$\DOS(\A)$, and for the cornucopia of examples from \cite{S}.
\end{ack}


\newcommand{\arxiv}[1]
{\texttt{\href{http://arxiv.org/abs/#1}{arxiv:#1}}}
\renewcommand{\MR}[1]
{\href{http://www.ams.org/mathscinet-getitem?mr=#1}{MR#1}}
\newcommand{\doi}[1]
{\texttt{\href{http://dx.doi.org/#1}{doi:#1}}}

\end{document}